\documentclass{amsart}
\usepackage{amssymb}
 \newtheorem{thm}{Theorem}[section]
 \newtheorem{cor}[thm]{Corollary}

 \theoremstyle{definition}
 
 \theoremstyle{remark}
 \newtheorem{rem}[thm]{Remark}
 
 \numberwithin{equation}{section}

\newcommand{\Q}{\mathbb Q}
\newcommand{\Z}{\mathbb Z}
\def\Gal{\operatorname{Gal}}

\begin{document}

%
%
%
%
%
%
%
%
%

\title[Noether's problem for cyclic groups]
 {On Noether's rationality problem\\ for cyclic groups over $\mathbb{Q}$}

\author[B. Plans]{Bernat Plans}

\address{%
Departament de Matem\`{a}tiques \\
Universitat Polit\`{e}cnica de Catalunya\\
Av. Diagonal, 647\\
08028 Barcelona}

\email{bernat.plans@upc.edu}

\thanks{Research partially supported by grant 2014 SGR-634 and grant MTM2015-66180-R}




\begin{abstract}
Let $p$ be a prime number. Let $C_p$, the cyclic group of order $p$, permute transitively a set of indeterminates $\{ x_1,\ldots ,x_p \}$. We prove that the invariant field $\Q(x_1,\ldots ,x_p)^{C_p}$ is rational over $\Q$ if and only if the $(p-1)$-th cyclotomic field $\Q(\zeta_{p-1})$ has class number one.
\end{abstract}

\maketitle
\section{Introduction}
Let a finite group $G$ act regularly on a set of indeterminates $\{x_1,\ldots ,x_n\}$ and let $k$ be a field. {\it Noether's problem for $G$ over $k$} asks whether the field extension $k(x_1,\ldots ,x_n)^G/k$ is rational, i.e. purely transcendental. 

The present note deals with Noether's problem for finite cyclic groups over the field of rational numbers. The reader is referred to \cite{Hoshi2015} for a brief survey of Noether's problem for abelian groups, including the most relevant references to work of Masuda, Swan, Endo, Miyata, Voskresenski,  Lenstra and others.

Let $P_{\Q}$ denote the set of prime numbers $p$ for which $\Q(x_1,\ldots ,x_p)^{C_p}/\Q$ is rational, where $C_p$ denotes the cyclic group of order $p$. 

Lenstra proved in \cite[Cor. 7.6]{Lenstra74} that $P_{\Q}$ has Dirichlet density $0$ inside the set of all prime numbers. Moreover, he suggested in \cite[p. 8]{Lenstra79} that $P_{\Q}$ could be finite and that perhaps coincides with the set 
$$R:=\{2,3,5,7,11,13,17,19,23,29,31,37,41,43,61,67,71\}.$$ 
It is known that $R\subseteq P_{\Q}$. This is a consequence of the fact that, by the main result in \cite{MasleyMontgomery1976}, $R$ is nothing but the set of prime numbers $p$ such that the $(p-1)$-th cyclotomic field $\Q(\zeta_{p-1})$ has class number one.

For prime numbers $p<20000$, some computational evidence in favour of the equality $P_{\Q}=R$ is given by Hoshi in \cite{Hoshi2015}. 

Our goal is to check the validity of Lenstra's suggestion. We prove:

\begin{thm}
\label{main}
$P_{\Q}=R$.
\end{thm}

From \cite[Cor. 3]{Lenstra79} and \cite[Prop. 4]{Lenstra79}, we get: 

\begin{cor}
\label{coro}
Let $n$ be a positive integer and let $C_n$ denote the cyclic group of order $n$. Then $\Q(x_1,\ldots ,x_n)^{C_n}/\Q$ is rational if and only if $n$ divides 
$$2^2\cdot3^m\cdot5^2\cdot7^2\cdot11\cdot13\cdot17\cdot19\cdot23\cdot29\cdot31\cdot37\cdot41\cdot43\cdot61\cdot67\cdot71,$$
for some $m\in \Z_{\geq 0}$.
\end{cor}

\section{Proof}
\begin{proof}[Proof of Thm. \ref{main}]
As has already been mentioned, the inclusion $R\subseteq P_{\Q}$ is known. See \cite[Prop. 3.4]{EndoMiyata1973}.
 
Let $p\in P_{\Q}$. This implies (actually, it is equivalent to) the existence of an element $\alpha \in \Z[\zeta_{p-1}]$ with norm $N_{\Q(\zeta_{p-1})/\Q}(\alpha)=\pm p$. See \cite[Thm. 3.1]{EndoMiyata1973}.

Thus, $\mathfrak{p}=(\alpha)$ is a principal prime ideal in $\Z[\zeta_{p-1}]$ above $(p)$. 

If $\Gal(\Q(\zeta_{p-1})/\Q)=\{\sigma_1,\ldots ,\sigma_m\}$, then we have the prime ideal decomposition 
$$(p)\Z[\zeta_{p-1}] = \sigma_1(\mathfrak{p}) \cdots \sigma_m(\mathfrak{p}).$$
Here $m=[\Q(\zeta_{p-1}):\Q]=\phi(p-1)$, where $\phi$ denotes Euler's totient function. Note that $(p)$ splits completely in $\Q(\zeta_{p-1})$, hence $\sigma_i(\mathfrak{p})\neq \sigma_j(\mathfrak{p})$ for $i\neq j$.

Now, a result of Amoroso and Dvornicich \cite[Cor. 2]{AmorosoDvornicich2000} ensures that
$$\frac{\log(p)}{\phi(p-1)}\geq 
\left\{ \begin{array}{ll}
            \dfrac{\log(5)}{12}, & \text{ for every } p, \\[.3cm]
            \dfrac{\log(7/2)}{8}, & \text{ for every } p\not\equiv 1 \pmod{7}.
           \end{array}\right.$$
It may be worth mentioning here that we are not assuming that $\Q(\zeta_{p-1})$ contains an imaginary quadratic subfield, eventhough this hypothesis is apparently used in the proof of \cite[Cor. 2]{AmorosoDvornicich2000}; in fact, if $\overline{\alpha}$ denotes the complex conjugate of $\alpha$, then the argument in \cite[Cor. 2]{AmorosoDvornicich2000} works whenever $(\alpha)\neq (\overline{\alpha})$, and this holds because $(p)$ splits completely in $\Q(\zeta_{p-1})$.

On the other hand, from a result of Rosser and Schoenfeld \cite[Thm. 15]{RosserSchoenfeld1962}, we also know that
$$\frac{\log(p)}{\phi(p-1)}<\frac{\log(p)}{p-1}\left( e^C \log(\log(p-1))+\frac{5}{2\log(\log(p-1))} \right),$$
where $C\approx 0.57721$ denotes Euler's constant.

If $f(p)$ denotes the right hand side of the above inequality, it is easily checked that $f(x)$ defines a decreasing function for, say, $x> 43$. Since $f(173)<\frac{\log(5)}{12}$, we conclude that $p<173$.

Once we restrict ourselves to prime numbers $p<173$, Hoshi's computations \cite{Hoshi2015} show that the only possible counterexamples to the inclusion $P_{\Q}\subseteq R$ are $59$, $83$, $107$ and $163$. 

Finally, each $p\in \{59,83,107,163\}$ satisfies 
$$p\not\equiv 1\pmod{7} \quad \text{ and } \quad \frac{\log(p)}{\phi(p-1)}<\frac{\log(7/2)}{8},$$
hence $p\notin P_{\Q}$.

\end{proof}

\begin{rem}
Let $n=p^r$ for some prime number $p\geq 5$. 

Lenstra proved \cite[Lemma 5]{Lenstra79} that $\Z[\zeta_{\phi(n)}]$ contains no element of norm $\pm p$ in the following cases:
\begin{itemize}
 \item[(i) ] $p\geq 11$ and $r\geq 2$.
 \item[(ii) ] $p\geq 5$ and $r\geq 3$.
\end{itemize}
Then, by \cite[Thm. 3.1]{EndoMiyata1973}, $\Q(x_1,\ldots ,x_n)^{C_n}/\Q$ cannot be rational in these cases \cite[Prop. 4]{Lenstra79}. 

Arguing as in the proof of Theorem \ref{main}, one can easily prove Lenstra's Lemma as follows. 

If $\alpha \in \Z[\zeta_{\phi(n)}]$ has norm $\pm p$, then $\mathfrak{p}=(\alpha)$ is a principal prime ideal above $(p)$ whose inertia degree over $(p)$ is $1$. Since $(p)$ splits completely in $\Z[\zeta_{p-1}]$, it must be $\mathfrak{p}\neq \overline{\mathfrak{p}}$. It follows that Amoroso and Dvornicich's result \cite[Cor. 2]{AmorosoDvornicich2000} applies and it ensures that
$$\frac{\log(p)}{\phi(\phi(n))} \geq \frac{\log(5)}{12}.$$ 
But it is readily seen that this inequality does not hold in cases (i) and (ii), just checking that:
\begin{itemize}
 \item[1) ] In case (i),  $\dfrac{\log(p)}{\phi(\phi(n))}\leq \dfrac{\log(p)}{2(p-1)}\leq \dfrac{\log(11)}{2\cdot 10}<\dfrac{\log(5)}{12} $.
 \item[2) ] In case (ii), $\dfrac{\log(p)}{\phi(\phi(n))}\leq \dfrac{\log(p)}{p(p-1)}\leq \dfrac{\log(5)}{5\cdot 4}<\dfrac{\log(5)}{12} $.
\end{itemize}
\end{rem}



\end{document}